\documentclass{article}

\usepackage[utf8]{inputenc}
\usepackage{lipsum}
\usepackage{url}
\usepackage{mathpazo}
\usepackage[T1]{fontenc}
\usepackage[english]{babel}
\usepackage{graphicx}

\usepackage[left=2.1cm, right=2.1cm, top=2.5cm]{geometry}

\usepackage{amsfonts}
\usepackage{amssymb}
\usepackage{amsthm}
\usepackage{amsmath}
\usepackage{mathrsfs}
\usepackage{epsfig}
\usepackage{palatino}
\usepackage{tikz,fp,ifthen}
\usepackage{float}

\newcommand{\pdfgraphics}{\ifpdf\DeclareGraphicsExtensions{.pdf,.jpg}\else\fi}
\usepackage{graphicx}

\usepackage{color}
\definecolor{refred}{rgb}{0.8,0,0}
\usepackage[colorlinks=true,linkcolor=blue,citecolor=refred,%
filecolor=blue,urlcolor=blue,pageanchor=true,plainpages=false,%
linktocpage]{hyperref}

\usepackage{mathtools}
\usepackage{comment}
\usepackage{bm}
\usepackage{mathtools}

\theoremstyle{plain}
\newtheorem{theorem}{Theorem}[section]
\newtheorem{lemma}[theorem]{Lemma}
\newtheorem{proposition}[theorem]{Proposition}

\newtheorem{corollary}[theorem]{Corollary}
\newtheorem{remark}[theorem]{Remark}
\newtheorem{definition}[theorem]{Definition}
\theoremstyle{definition}
\theoremstyle{remark}
\numberwithin{equation}{section}

\newcommand{\as}{{\mathcal A}}
\newcommand{\ET}{{\rm E}_{\rm T}}
\newcommand{\EB}{{\rm E}_{\rm B}}
\newcommand{\hs}{{\mathcal H}}
\newcommand{\ks}{{\mathcal K}}
\newcommand{\cs}{{\mathcal C}}
\newcommand{\gs}{{\mathcal G}}
\newcommand{\fs}{{F}}
\newcommand{\leb}{{\mathcal L}}
\newcommand{\ds}{{\mathcal D}}
\newcommand{\ms}{{\mathcal M}}
\newcommand{\ns}{{\mathcal N}}
\newcommand{\bs}{{\mathcal B}}
\newcommand{\vub}{{\mathcal V}}
\newcommand{\Ps}{{\mathcal P}}
\newcommand{\dis}{{\mathcal D}}
\newcommand{\Es}{{\mathcal E}}
\newcommand{\ts}{{\mathcal T}}
\newcommand{\rs}{{\mathcal R}}
\newcommand{\dst}{{\mathcal D}}
\newcommand{\qs}{{\mathcal Q}}
\newcommand{\Ws}{{\mathcal W}}
\newcommand{\vc}[1]{\bm{#1}}
\newcommand{\sob}[2]{{W}^{#1}({#2})}
\newcommand{\crs}{\mathcal{A}}
\newcommand{\loo}{{\Lambda[w]}}
\newcommand{\abs}[1]{\left\lvert#1\right\rvert}

\newcommand{\R}{{\mathbb R}}
\newcommand{\C}{{\mathbb C}}
\newcommand{\B}{{\mathbb B}}
\newcommand{\N}{{\mathbb N}}
\newcommand{\Z}{{\mathbb Z}}
\newcommand{\LL}{{\mathbb L}}
\newcommand{\MM}{{\mathbb M}}
\newcommand{\W}{{\mathbb W}}
\newcommand{\T}{{\mathbb S}^{\rm p}}
\newcommand{\G}{{\mathbb G}}

\newcommand{\Mat}{{\rm M}^{N\times N}}
\newcommand{\Mskew}{{\rm M}^{N\times N}_{\rm skew}}
\newcommand{\Msym}{{\rm M}^{N\times N}_{\rm sym}}
\newcommand{\Mdev}{{\rm M}^{N\times N}_{D}}
\newcommand{\Mcurl}{{\rm M}^{N\times N}}
\newcommand{\Mhigher}{{\rm M}^{N\times N\times N}_D}
\newcommand{\Mhigherbis}{{\rm M}^{N\times N\times N}}

\newcommand{\Ba}{B_1(0)}
\newcommand{\Bb}{\overline{B}_1(0)}
\newcommand{\Bbg}{\overline{B}_h(0)}
\newcommand{\tint}[1]{{\rm int}(#1)}
\newcommand{\epi}[1]{{\rm epi}(#1)}
\newcommand{\dom}[1]{{\mathcal Dom}(#1)}
\newcommand{\diam}[1]{{\mathcal Diam}(#1)}

\newcommand{\Om}{\Omega}
\newcommand{\Omb}{\overline{\Omega}}
\newcommand{\Omt}{\tilde{\Omega}}
\newcommand{\Omp}{\Omega'}
\newcommand{\Ombh}{\overline{\Omega_h}}
\newcommand{\bv}{BV(\Omega)}
\newcommand{\Cc}{C_c^1(\Omega)}
\newcommand{\deli}[2]{L^{1,2}(#1 \setminus #2)}
\newcommand{\LD}[2]{LD(#1 \setminus #2)}
\newcommand{\Ni}{\Gamma}

\newcommand{\norm}[1]{||#1||}
\newcommand{\nL}[1]{||\nabla #1||_a}
\newcommand{\nA}[1]{||E #1||_A}

\newcommand\aplim{\mathop{\rm ap\,lim}}
\newcommand{\weakst}{\stackrel{\ast}{\rightharpoonup}}
\newcommand{\weakstloc}{{\stackrel{\ast}{\rightharpoonup}}_{loc}}
\newcommand{\weak}{\rightharpoonup}
\newcommand{\wlystar}{$\text{weakly}^*$\;}
\newcommand{\wstar}{$\text{weak}^*$\;}
\def\wtostar{\stackrel{*}{\rightharpoonup}}

\newcommand{\hn}{\hs^{N-1}}
\newcommand{\hu}{\hs^1}
\newcommand{\esbd}[1]{\partial^* \!#1}
\newcommand{\per}[2]{{\rm P}(#1,#2)}
\newcommand{\rettificabili}{\mathcal{R}}

\newcommand{\E}{\mathbf E}
\def\Div{\textup{div}\,}
\newcommand{\Gb}{\mathbf G}
\newcommand{\ab}{\mathbf A}
\newcommand{\Ee}{\mathbf E^{\rm e}}
\newcommand{\dotEe}{\dot{\mathbf E}^{\rm e}}
\newcommand{\Ep}{\mathbf E^{\rm p}}
\newcommand{\dotEp}{\dot{\mathbf E}^{\rm p}}
\newcommand{\Hb}{\mathbf H}
\newcommand{\Wb}{\mathbf W}
\newcommand{\Tb}{\mathbf T}\newcommand{\Tbp}{{\mathbf T}^{\rm p}}
\newcommand{\Fb}{\mathbf F}
\newcommand{\qb}{\mathbf q}
\newcommand{\pb}{\mathbf p}
\newcommand{\eb}{\mathbf e}
\newcommand{\pscal}[2]{\langle #1, #2 \rangle}
\newcommand{\curl}[1]{{\rm curl } #1}
\newcommand{\K}{{\mathbb K}^{\rm p}}
\newcommand{\Ken}{{\mathbb K}^{\rm p}_{\rm en}}
\newcommand{\Kdiss}{{\mathbb K}^{\rm p}_{\mathcal Diss}}
\newcommand{\e}{\varepsilon}
\newcommand{\F}{F}
\newcommand{\n}{\textbf{\em n}}
\renewcommand\H{\mathcal{H}}

\newcommand{\tens}[1]{\mathsf{#1}}
\newcommand{\tpl}{\mathsf{w}}

\begin{document}
\pdfgraphics 

\title{Geometric invariants of non-smooth framed curves}

\author{\textsc{Giulia Bevilacqua} $^1$\thanks{\href{mailto:giulia.bevilacqua@dm.unipi.it}{\texttt{giulia.bevilacqua@dm.unipi.it}}}\,\,\,$-$\,\, \textsc{Luca Lussardi} $^2$\thanks{\href{mailto:luca.lussardi@polito.it}{
\texttt{luca.lussardi@polito.it}}} \,\,\,$-$\,\, \textsc{Alfredo Marzocchi} $^3$\thanks{\href{mailto:alfredo.marzocchi@unicatt.it}{
\texttt{alfredo.marzocchi@unicatt.it}}}\bigskip\\
\normalsize$^1$ Dipartimento di Matematica, Università di Pisa, Largo Bruno Pontecorvo 5, I–56127 Pisa, Italy.\\
\normalsize$^2$ DISMA ``G.L.\,Lagrange'', Politecnico di Torino, c.so Duca degli Abruzzi 24, I-10129 Torino, Italy.\\
\normalsize$^3$ Dipartimento di Matematica e Fisica ``N. Tartaglia", Università Cattolica del Sacro Cuore,\\
\normalsize via della Garzetta 48, I-25133 Brescia, Italy.
}

\maketitle

\begin{abstract}
\noindent We compare the Serret-Frenet frame with a {\em relatively parallel adapted frame} (RPAF) introduced by Bishop \cite{BISHOP} to parametrize $W^{2,2}$-curves. Next, we derive the geometric invariants, curvature and torsion, with the RPAF associated to the curve. Finally, we discuss applications of the two approaches in variational problems.
\end{abstract}

\bigskip



\bigskip

\section{Introduction}
The study of curvature and torsion of spatial curves is contained in differential geometry where high regularity on the curve is customary. In such a setting, the curvature $\kappa$ and torsion $\tau$ generate the {\em Serret-Frenet frame}, which is an orthonormal basis along the curve, where $\tau$ is defined only in points with $\kappa \neq 0$. The Serret-Frenet system takes the form
$$
\left\{
\begin{aligned}
{\vc t}'&=\kappa\,\vc n,\\
{\vc n}'&=-\kappa\,\vc t-\tau\vc b,\\
{\vc b}'&=\tau\vc n,
\end{aligned}
\right.
$$
where $\vc t$ is tangent to the curve parametrized by the arc-length, $\vc n$ is the principal normal and $\vc b$ is the binormal.

Variational analysis of elastic curves has been widely developed by Langer and Singer \cite{langer1984knotted, langer1984total, langer1985curve}. We also mention a recent interesting research on elastic networks by Mantegazza et al. \cite{MPP}, where the authors consider essentially the Euler elasticae.
Nevertheless, to extend the study to more general energy functionals, one may need to introduce a suitable weak notion of curvature and torsion defined everywhere.

More than 40 years ago, Bishop \cite{BISHOP} introduced another way to frame a curve. Precisely, he defined the {\em relatively parallel adapted frame} (RPAF) as an orthonormal basis along the curve such that normal vectors to the curve have tangential derivatives. This requires less regularity than in the classical case, for instance it can be set for a $W^{2,2}$-curve. Precisely, a RPAF is given by
$$
\left\{
\begin{aligned}
{\vc t}'&=u_2\,\vc d_1-u_1\vc d_2,\\
{\vc d}'_1&=-u_2\, \vc t,\\
{\vc d}'_2&=u_1 \,\vc t,
\end{aligned}
\right.
$$
where $\vc t$ is tangent to the curve and $(\vc d_1, \vc d_2)$ are orthogonal to $\vc t$. We notice that the coefficients $u_1,u_2$ are not geometric invariants of the curve.

Another way to set the problem is the {\em framed curve approach}, which has been reconsidered by Schuricht et al. \cite{GMSM, S} to study elastic rods. We also mention \cite{GF}, where the RPAF has been obtained within the theory of Cosserat rods.
Framed curves have been recently used to study variational problems related to elastic curves. We refer to \cite{FHMP, FHMP2}, where they derived a corrected version of the Sadowsky functional within the theory of elastic ribbons and to \cite{GLF, BLM1, BLM2, BLM3}, where the authors investigated equilibrium shapes of a system in which a closed flexible filament is spanned by a liquid film. Other results can be found in \cite{BLMPRSA, BBLM}, where the framed curve approach has been employed to derive first-order necessary conditions for minimizers.

In this paper, in Section \ref{sec:RPAF}, we recall the Bishop framework introducing the relatively parallel adapted frames (RPAF) and we compare them with the Serret-Frenet frames in the smooth case. Next, in Section \ref{sec:invariant}, we derive the geometric invariants $\kappa$ and $\tau$ as functions of the coefficients $u_1, u_2$ of the RPAF. Finally, in Section \ref{sec:conclusion}, we conclude that the framed curve approach seems to be useful to study variational problems involving energy functionals, while the RPAF approach is less suitable for that. However, the framed curve approach requires more than a $W^{2,2}$-curve, whereas in the RPAF one, curvature and torsion are well-defined for any $W^{2,2}$-curve.

\section{Relatively adapted frames along a curve}
\label{sec:RPAF}
Let us consider a curve $\vc x \in \cs^2((0, L);\R^3)$, where $L>0$, parametrized by the arc-length $s$. Following Bishop \cite{BISHOP}, we start recalling the notion of relatively parallel fields along $\vc x$. Let $\vc t = \vc x'$ be the unit tangent vector to $\vc x$ and let $\vc d$ in the plane perpendicular to $\vc t$.  We say that $\vc d$ is {\em relatively parallel field along $\vc x$} if $\vc d'=c\vc t$ for some constant $c$ (see Figure \ref{fig1}).


\begin{figure}[ht]
\centering
\includegraphics[width=11cm]{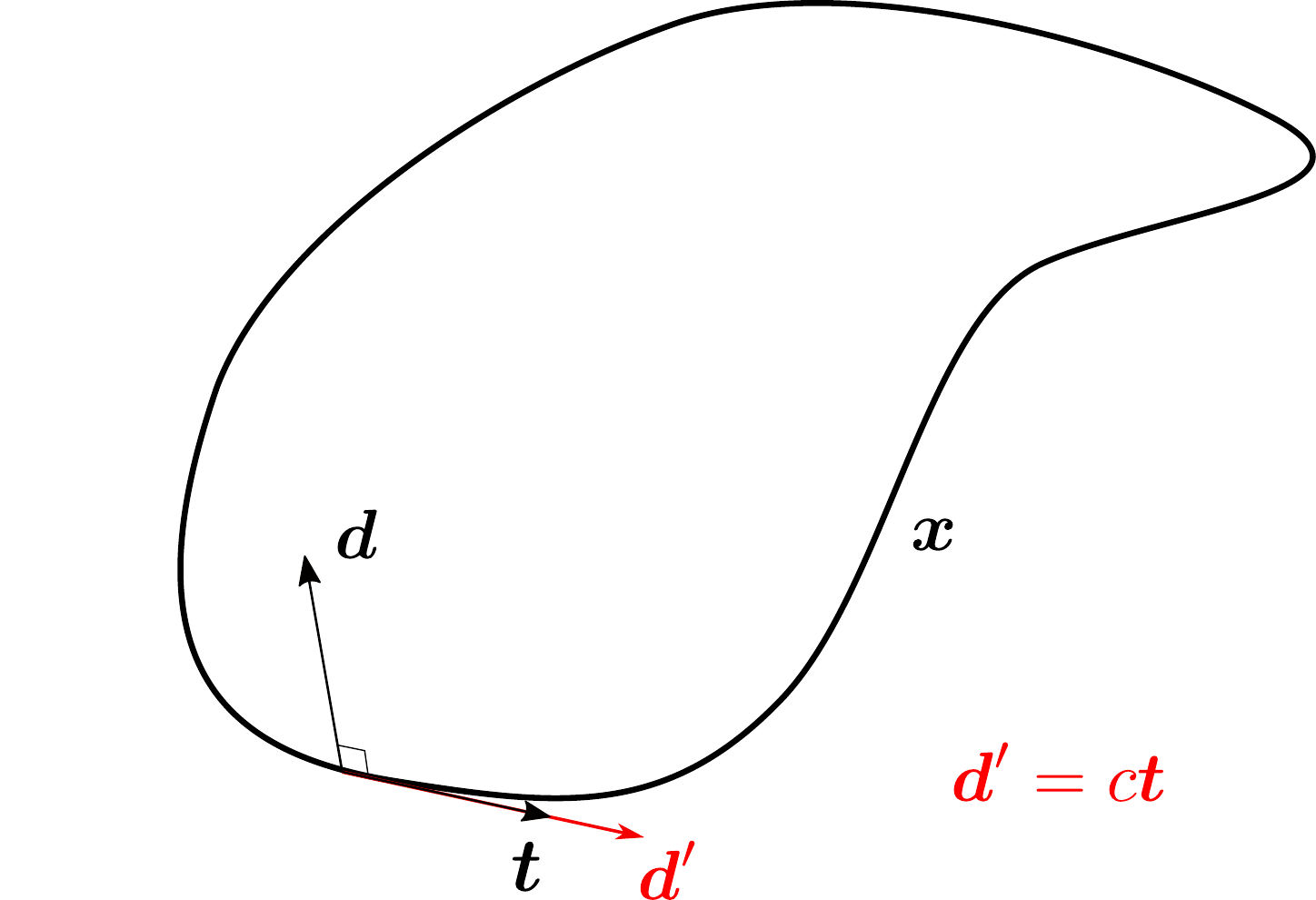}
\caption{The field $\vc d$ is {\em relatively parallel} along the curve $\vc x$: $\vc d' = c \vc t$ where $c$ is a constant.}\label{fig1}
\end{figure}
\begin{remark}
    \label{existence_relatively_parallel}
By means of a standard parallel transport argument, see \cite[Thm.\,1]{BISHOP} for details, we can ensure that there exists at least one relatively parallel field along $\vc x$.
\end{remark}
A {\it relatively parallel adapted frame along $\vc x$}, briefly {\it RPAF along $\vc x$}, is an orthonormal frame $\{\vc t,\vc d_1,\vc d_2\}$ along the curve such that both $\vc d_1,\vc d_2$ are relatively parallel fields along $\vc x$. As a consequence, there are smooth coefficients $u_i$, $i=1,2$, such that
\begin{equation}
    \label{eq:curva_antman}
\left\{
\begin{aligned}
{\vc t}'&=u_2\,\vc d_1-u_1\vc d_2,\\
{\vc d}'_1&=-u_2\, \vc t,\\
{\vc d}'_2&=u_1 \,\vc t.
\end{aligned}
\right.
\end{equation}
Clearly, any orthonormal frame satisfying \eqref{eq:curva_antman} is a RPAF along any curve with $\vc t$ as tangent vector.

We notice that the system \eqref{eq:curva_antman} is similar to the classical Serret-Frenet system which can be defined for a $\cs^3$-curve. Precisely, if $\vc x\in \cs^3((0,L);\R^3)$ is parametrized by the arc-length, then $\vc t=\vc x'$. Assume $|\vc t'|\ne 0$ everywhere and let $\kappa = \abs{\vc t'} $ the {\em (positive) curvature} of $\vc x$. We set
\[
\vc n=\frac{\vc t'}{|\vc t'|}, \qquad \vc b=\vc t \times \vc n,
\]
usually referred as the {\it principal normal} and the {\it binormal} respectively. Directly, we have
$$
\vc b' = \tau \vc n,
$$
for some coefficient $\tau$ which is called the {\em torsion} of the curve $\vc x$.
Hence, the Serret-Frenet orthonormal frame $\{\vc t, \vc n, \vc b\}$ satisfies the system 
\begin{equation}
    \label{eq:def_SF}
\left\{
\begin{aligned}
{\vc t}'&=\kappa\,\vc n,\\
{\vc n}'&=-\kappa\,\vc t-\tau\vc b,\\
{\vc b}'&=\tau\vc n.
\end{aligned}
\right.
\end{equation}
We notice that system \eqref{eq:def_SF} is well-defined since $|\vc t'| \ne 0$. 
The moving orthonormal frame $\{\vc t, \vc n, \vc b\}$ is not an RPAF along $\vc x$ since both $\vc n$ and $\vc b$ are not relatively parallel along $\vc x$. Nevertheless, the Serret-Frenet system contains important information about $\vc x$. Indeed, in \eqref{eq:def_SF} the coefficients $\kappa$ and $\tau$ are {\it geometric invariants} of $\vc x$: the curvature and the torsion of $\vc x$ do not depend on the parametrization of $\vc x$.

Both RPAF and Serret-Frenet frames are special cases of the general frame along the curve, which looks like
\begin{equation}
    \label{eq:sistema_tot}
\left\{
\begin{aligned}
{\vc t}'&=u_2\,\vc d_1-u_1\vc d_2,\\
{\vc d}'_1&=-u_2\, \vc t + u_3 \vc d_2,\\
{\vc d}'_2&=u_1 \,\vc t - u_3 \vc d_1,
\end{aligned}
\right.
\end{equation}
where $u_1, u_2$ and $u_3$ are smooth coefficients.

To deal with non-smooth curves, the idea is to prescribe the coefficients $u_i\in L^2(0,L)$, ($u_1, u_2$ are called {\it flexural densities}, while $u_3$ is the {\it twist density}), and look for a solution of \eqref{eq:sistema_tot}. More precisely, let us fix the following initial conditions
\begin{equation}\label{eq:clamping}
\vc t(0) = \vc t^0, \quad \vc d_1(0) = \vc d_1^0, \quad \vc d_2(0) = \vc d_2^0,
\end{equation}
such that $\{{\vc t}^0,{\vc d_1}^0,{\vc d_2}^0\}$ is an orthonormal basis in $\R^3$. By classical results  \cite{hartman}, there exists a unique orthonormal frame $\{\vc t, \vc d_1,\vc d_2\} \in (W^{1,2}((0,L);\R^3))^3$ satisfying \eqref{eq:sistema_tot} and \eqref{eq:clamping}. By integration we can therefore reconstruct the curve $\vc x$ as
\[
\vc x(s)=\vc t^0+\int_0^s\vc t(r)\,dr.
\]
In particular, we get $\vc x\in W^{2,2}((0,L);\R^3)$. 
This approach has been introduced and developed by Gonzalez et al.\,(see \cite{GMSM} and \cite{S}) and it has been defined the {\em framed curve approach}. In the case where $u_3 = 0$, we call again $\{\vc t, \vc d_1,\vc d_2\}$ a RPAF along the curve $\vc x$.

From now on, we will focus only on systems \eqref{eq:curva_antman} and \eqref{eq:def_SF} and we compare them. A natural question arises: are $\kappa$ and $\tau$ related with the coefficients $u_1,u_2$ of an RPAF along $\vc x$? In the next Lemma we show the relation between $u_i$'s and $\kappa,\tau$.

\begin{lemma}
\label{lemma:sF_to_RPAF}
Let $u_i\colon [0,L] \to \R$, $i=1,2$, be such that the unique solution $\vc x$ of \eqref{eq:curva_antman} and \eqref{eq:clamping} is of class $\cs^3((0,L);\R^3)$. Assume that $|\vc t'|=|\vc x''|\ne 0$ everywhere. Let $\kappa,\tau$ be the curvature and the torsion of $\vc x$. There exists $\vartheta \in \cs^1([0,L])$ such that 
\begin{equation}
\label{relation}
\left \{
\begin{aligned}
{\vartheta}' &=\tau \\
u_1 &= \kappa\sin \vartheta\\
u_2 &= \kappa \cos \vartheta.
\end{aligned}
\right.
\end{equation}
In particular, 
\[
\sqrt{u_1^2+u_2^2}=\kappa.
\]
\end{lemma}
\begin{proof}
Since the curve is regular, both systems
\begin{equation}
    \label{eq:curva_antman_theorema}
\left\{
\begin{aligned}
{\vc t}'&=u_2\,\vc d_1-u_1\vc d_2\\
{\vc d}'_1&=-u_2\, \vc t\\
{\vc d}'_2&=u_1 \,\vc t\\
\end{aligned}
\right.
\end{equation}
and 
\begin{equation}
    \label{eq:Serret-Frenet_teorema}
\left\{
\begin{aligned}
{\vc t}'&=\kappa\,\vc n\\
{\vc n}'&=-\kappa\,\vc t-\tau\vc b\\
{\vc b}'&=\tau\vc n.
\end{aligned}
\right.\, 
\end{equation}
hold.

Comparing \eqref{eq:Serret-Frenet_teorema}${}_1$ and \eqref{eq:curva_antman_theorema}${}_1$, it follows
\begin{equation}
    \label{eq:SF-RPAF}
        \kappa\vc n = \vc t' =u_2\,\vc d_1-u_1\vc d_2,
\end{equation}
whence, remembering that $\vc d_1$ and $\vc d_2$ are orthonormal, 
\begin{equation}
    \label{eq:def_u1u2}
    \begin{aligned}
    &\kappa^2=u_1^2+u_2^2.
\end{aligned}
\end{equation}
This suggests to set
\begin{equation}
    \label{eq:u1_u2_theta}
    \begin{aligned}
        &u_1=\kappa\sin\vartheta,&&&u_2=\kappa\cos\vartheta,
    \end{aligned}
\end{equation}
for a suitable function $\vartheta \in \cs^1((0,L))$.

To relate $\vartheta$ to the torsion $\tau$, we notice that since $\kappa \neq 0$ and $\vc x \in \cs^3((0,L);\R^3)$, we get $\kappa \in \cs^1((0,L))$. Hence, differentiating \eqref{eq:u1_u2_theta} and using \eqref{eq:SF-RPAF}, we obtain
\begin{align}
\label{eq:deriv-u_1}
&u_1'=\kappa'\sin\vartheta+\kappa\vartheta'\cos\vartheta=\frac{\kappa'}{\kappa}u_1 + u_2\vartheta'\\
\label{eq:deriv-u_2}
    &u_2'=\kappa'\cos\vartheta-\kappa\vartheta'\sin\vartheta=\frac{\kappa'}{\kappa}u_2 - u_1\vartheta'.
\end{align}
    
On the other hand, differentiating \eqref{eq:SF-RPAF}, it follows
$$
\kappa'\vc n+\kappa{\vc n}'=u_2'\vc d_1+u_2{\vc d_1}'-u_1'\vc d_2-u_1{\vc d_2}'.
$$
Substituting the expressions of ${\vc n}',{\vc d_1}', {\vc d_2}'$ from \eqref{eq:curva_antman_theorema} and \eqref{eq:Serret-Frenet_teorema} respectively, and using \eqref{eq:SF-RPAF}, \eqref{eq:def_u1u2}, \eqref{eq:deriv-u_1} and \eqref{eq:deriv-u_2}, one  easily gets
$$
\kappa'\vc n-\kappa\tau\vc b=-\vartheta'(u_1\vc d_1+u_2\vc d_2)+\frac{\kappa'}{\kappa}\left(u_2\vc d_1-u_1\vc d_2\right) = -\vartheta'(u_1\vc d_1+u_2\vc d_2) + \frac{\kappa'}{\kappa} \kappa \vc n.
$$
This simplifies into
$$
\kappa\tau\vc b=\vartheta'(u_1\vc d_1+u_2\vc d_2).
$$
Squaring this relation and using \eqref{eq:def_u1u2}, it follows immediately $\tau=\pm\vartheta'$, which is the thesis.
\end{proof}

\begin{remark}
\label{remark:rotazione}
    Equation \eqref{eq:SF-RPAF}, together with \eqref{eq:u1_u2_theta}, implies for $\kappa\neq0$
    $$\vc n=\cos\vartheta\vc d_1-\sin\vartheta\vc d_2$$
    and since $\vc b$ is perpendicular to $\vc n$, it is immediate to check that
    $$\vc b=\sin\vartheta\vc d_1+\cos\vartheta\vc d_2.$$
    Therefore, $\vartheta(s)$ is, for every $s$, the angle of a rotation $\tens R(s)$ in the plane perpendicular to $\vc t(s)$ at $\vc x(s)$, which can be thought as a space rotation around $\vc t(s)$. i.e. leaving the tangent vector fixed: $\tens R(s)\vc t(s)=\vc t(s)$.
\end{remark}

\begin{remark}
Lemma \ref{lemma:sF_to_RPAF} shows that, assuming smoothness of the curve, the curvature and the torsion are related to the flexural densities $u_1$ and $u_2$ through the twist $\vartheta$ of the moving frame. Nevertheless, from the system \eqref{relation}, it is clear that $u_1$ and $u_2$ are not geometric invariants of the curve.
\end{remark}



In order to extract geometric invariants of a curve of class $W^{2,2}$ from a RPAF along it the we need to understand better the ``degrees of freedom'' of the RPAF. The next proposition is only stated in \cite{BISHOP}. 

\begin{figure}[h!]
\centering
\includegraphics[width=11cm]{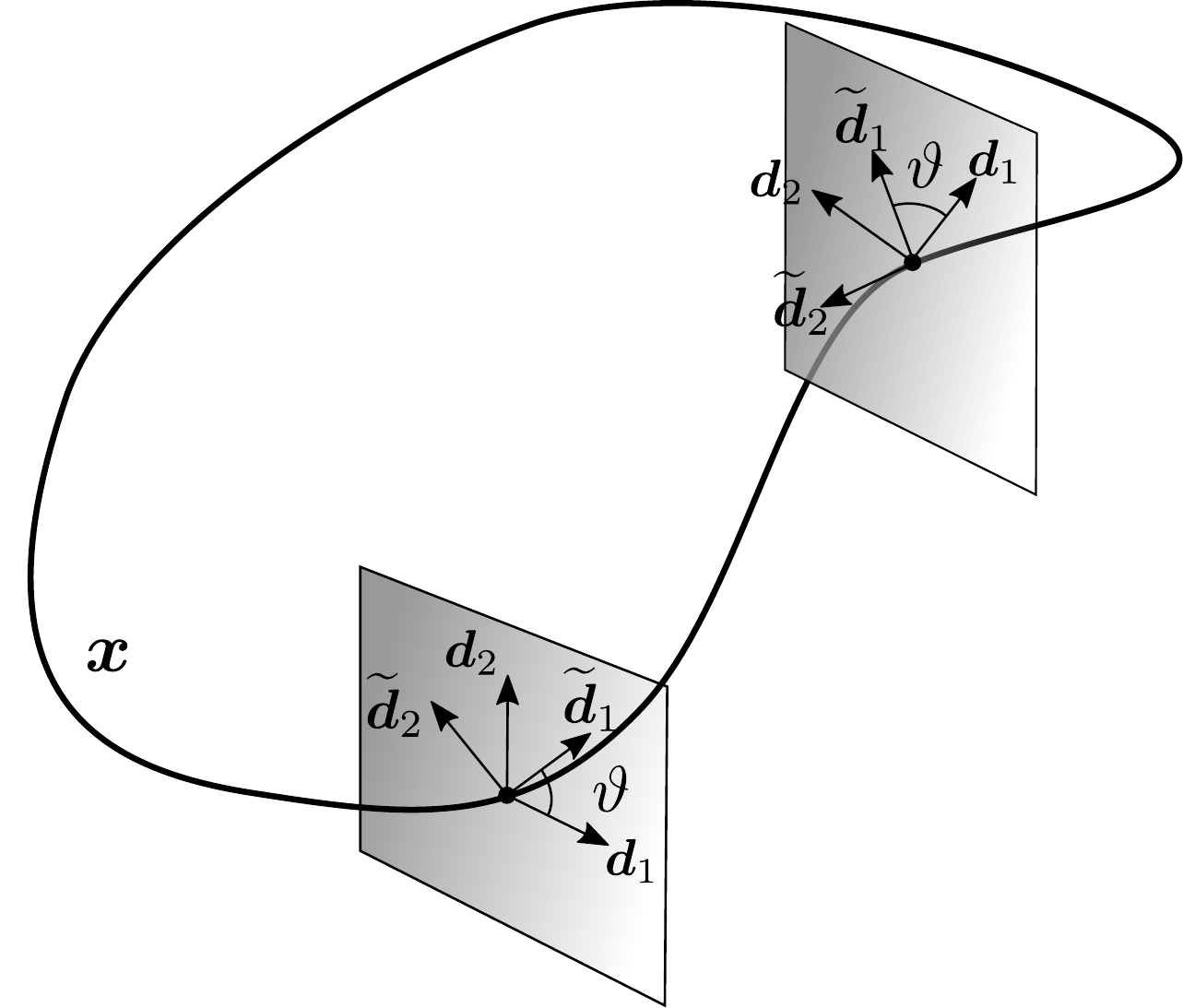}
\caption{Two RPAFs along $\vc x$: they differe by a constant angle of rotation $\vartheta$.}\label{fig3}
\end{figure}
 
\begin{proposition}
\label{prop:esistenzaRPAF}
If $\{\vc t, \vc d_1,\vc d_2\}$ is a RPAF, then the totality of RPAFs consists of frames of the form 
\[
\tens{R}(\vc t,\vc d_1, \vc d_2)
\]
where $\tens{R}$ is the rotation introduced in Remark \ref{remark:rotazione} and it is independent of $s$, see Figure \ref{fig3}.
\end{proposition}

\begin{proof}
Obviously, if $(\vc t,\vc d_1, \vc d_2)$ is an RPAF and $\tens{R}$ does not depend on $s$, then $\tens{R}(\vc t,\vc d_1, \vc d_2)$ is an RPAF.

Let $\{\vc t, \vc d_1, \vc d_2\}$ and $\{\vc t, \tilde{\vc d}_1, \tilde{\vc d}_2\}$ be two RPAFs defined on the curve $\vc x$. Then, necessarily in the plane perpendicular to $\vc t$, we have $(\vc d_1, \vc d_2)$ are related to $(\tilde{\vc d}_1, \tilde{\vc d}_2)$ through a rotation which depends on the applied point $s$, namely
\begin{equation}
\label{eq:legame_d_tilded}
    \begin{pmatrix}
\tilde{\vc d}_1 \\
\tilde{\vc d}_2
\end{pmatrix} = \tens{R} \begin{pmatrix}
{\vc d}_1 \\
{\vc d}_2 
\end{pmatrix} 
= \begin{pmatrix}
\cos \vartheta & -\sin\vartheta\\
\sin\vartheta & \cos\vartheta
\end{pmatrix} \, \begin{pmatrix}
{\vc d}_1 \\
{\vc d}_2
\end{pmatrix},
\end{equation}
where $\tens{R}$ is an rotation matrix which in general depends on $s$. We want to prove that actually $\tens R$ does not depend on $s$. Using the expression of $(\tilde{\vc d}_1,\tilde{\vc d}_2)$ in Eq. \eqref{eq:legame_d_tilded}, we get
$$
\begin{aligned}
&\tilde{\vc d}_1' = - \sin \vartheta \vartheta'(s) \vc d_1 + \cos\vartheta \vc d'_1 - \vartheta' \cos\vartheta \vc d_2(s) - \sin \vartheta \vc d_2',\\
&\tilde{\vc d}_2' = \cos \vartheta \vartheta' \vc d_1 + \sin\vartheta \vc d'_1 - \vartheta' \sin\vartheta \vc d_2(s) + \cos \vartheta\vc d_2'.
\end{aligned}
$$
Since the couple $(\vc d_1, \vc d_2)$ satisfies Eq. \eqref{eq:curva_antman}, the above expressions simplify into
$$
\begin{aligned}
&\tilde{\vc d}_1' = \underbrace{\vartheta'\left[-\sin \vartheta \vc d_1 -\cos\vartheta \vc d_2\right]}_{\mathcal{A}_1} + \cos\vartheta u_2 \vc t - \sin \vartheta u_1 \vc t,\\
&\tilde{\vc d}_2' = \underbrace{\vartheta'\left[\cos \vartheta \vc d_1 -\sin \vartheta \vc d_2\right]}_{\mathcal{A}_2} - \sin\vartheta u_2 \vc t + \cos \vartheta u_1 \vc t.
\end{aligned}
$$
Since both $\{\vc t,\vc d_1, \vc d_2\}$ and $\{\vc t, \tilde{\vc d}_1, \tilde{\vc d}_2\}$ were chosen to be RPAFs, i.e. the components of the derivatives of $\tilde{\vc d}_1'$ and $\tilde{\vc d}_2'$ are allowed only along the tangential direction $\vc t$, this implies that both $\mathcal{A}_1$ and $\mathcal{A}_2$ have to be zero. Hence, $\vartheta' = 0$ which gives the thesis.
\end{proof}

\begin{figure}[h!]
\centering
\includegraphics[width=11cm]{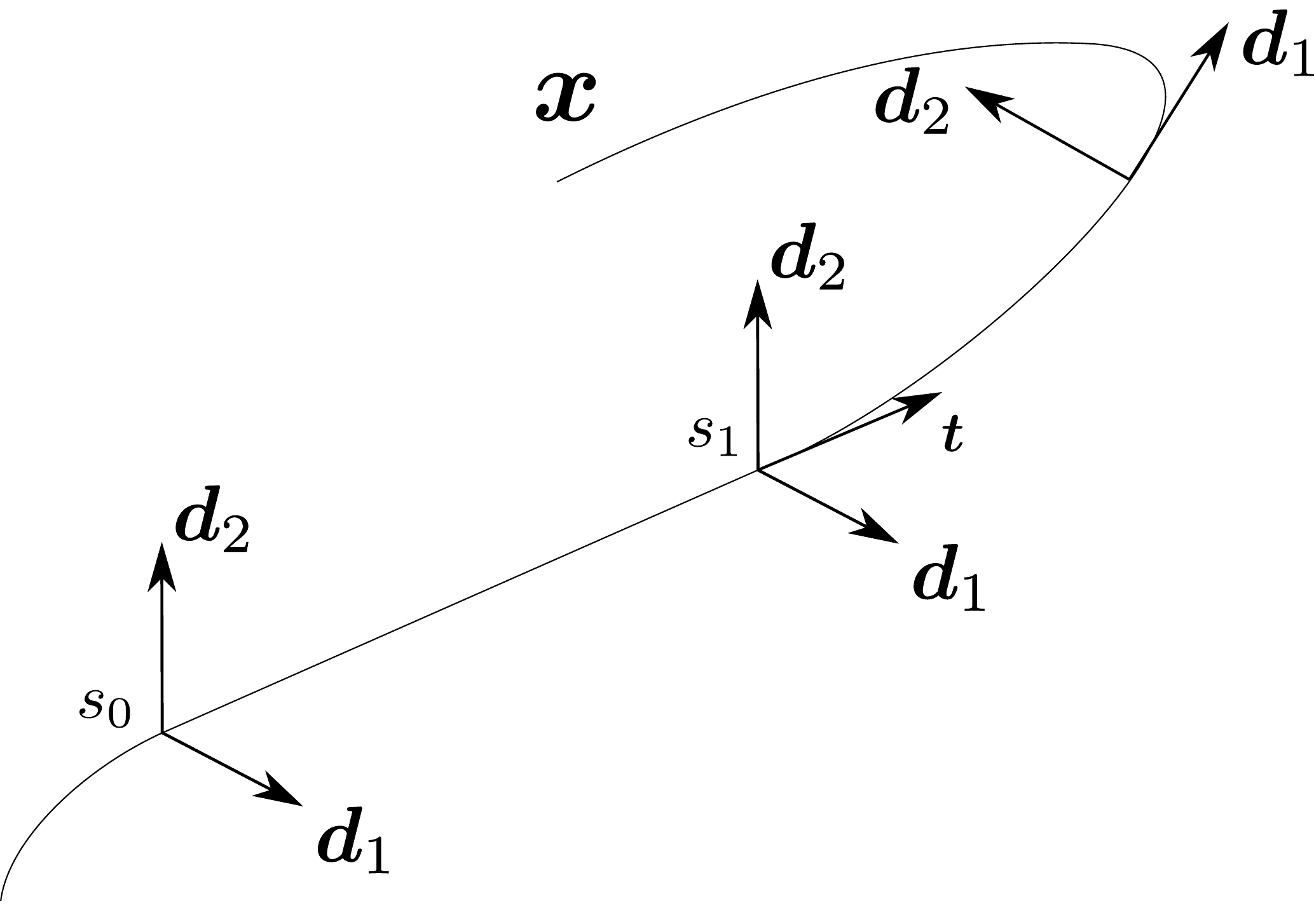}
\caption{Along a straight-line segment the twist of a RPAF is not arbitrary.}\label{fig2}
\end{figure}

\begin{remark}
    The definition of a RPAF and the Proposition \ref{prop:esistenzaRPAF} can be easily adapted to a curve in $\R^n$ for any $n >3$.
\end{remark}

An immediate consequence of Proposition \ref{prop:esistenzaRPAF} is the following remark.
\begin{remark}
    If $\vc x$ has a straight-line piece then a RPAF along such a piece of the curve must be constant (see Figure \ref{fig2}).
\end{remark}

\section{From the unit tangent field to geometric invariants}
\label{sec:invariant}
In this section we prescribe $\vc t\in W^{1,2}((0,L);\R^3)$ with $|\vc t|=1$ everywhere, so that, by integration, we obtain a curve $\vc x$, parametrized by the arc-length. We want to define directly geometric invariants of $\vc x$.

We begin with the following proposition. 

\begin{proposition}
    \label{th:integral}
Let $\vc d_1^0,\vc d_2^0 \in \R^3$ be such that $\{\vc t(0),\vc d_1^0,\vc d_2^0\}$ is an orthonormal basis in $\R^3$. Then the integral equations
\begin{align}
\label{eq:k1}
u_1(s) &= -\vc d_2^0 \cdot \vc t'(s) - \int_0^s u_1(r) \vc t(r ) \cdot\vc t'(s) \, dr,\\
\label{eq:k2}
u_2(s)&= \vc d_1^0 \cdot \vc t'(s) - \int_0^s u_2(r) \vc t(r ) \cdot\vc t'(s) \, dr.
\end{align}
have a unique solution $u_1,u_2 \in L^2(0,L)$. Moreover, if we set 
\begin{equation}\label{ds}
\vc d_1(s)=\vc d_1^0-\int_0^su_2(r)\vc t(r)\,dr, \quad \vc d_2(s)=\vc d_2^0+\int_0^su_1(r)\vc t(r)\,dr
\end{equation}
then $\{\vc t,\vc d_1,\vc d_2\}$ is an orthonormal basis in $\R^3$ and it is the unique solution of \eqref{eq:curva_antman}-\eqref{eq:clamping}.
\end{proposition}
\begin{proof}
The integral equations \eqref{eq:k1} and \eqref{eq:k2} are Volterra integral equations of the second kind with kernel in $L^2$. Applying \cite[Sec.\,1.5]{T} we get existence and uniqueness of solutions $u_1,u_2\in L^2(0,L)$. Then, if we define $\vc d_1,\vc d_2$ by means of \eqref{ds} we get $\vc d_1'=-u_2\vc t$ and $\vc d_2'=u_1\vc t$. At this point, from \eqref{eq:k1} - \eqref{eq:k2}, we obtain
\[
u_1=-\vc d_2 \cdot \vc t', \quad u_2=\vc d_1 \cdot \vc t'.
\]
As a consequence
\[
(\vc t \cdot \vc d_1)'=\vc t' \cdot \vc d_1+\vc t \cdot \vc d_1'=u_2-u_2=0, \quad (\vc t \cdot \vc d_2)'=\vc t' \cdot \vc d_2+\vc t \cdot \vc d_2'=-u_1+u_1=0.
\]
This means that $\vc t \cdot \vc d_i=\vc t(0) \cdot \vc d_i^0=0$ for $i=1,2$.
Furthermore, we also easily obtain 
\[
(\vc d_1 \cdot \vc d_1)'=(\vc d_2 \cdot \vc d_2)'=(\vc d_1 \cdot \vc d_2)'=0.
\]
Hence, $\{\vc t,\vc d_1,\vc d_2\}$ is an orthonormal basis in $\R^3$. Finally, $\vc t'=u_2\vc d_1-u_1\vc d_2$ and this yields the conclusion.
\end{proof}


To define geometric invariants of $\vc x$, following Giusteri et al. \cite{GF}, we introduce 
\begin{equation}
\label{legame_ktau}
u: = u_2 + {\rm i} u_1
\end{equation}
where $u_1$ and $u_2$ are the solutions of \eqref{eq:k1} and \eqref{eq:k2} respectively, having fixed the initial data $\vc d_1^0$ and $\vc d_2^0$. In order to simplify the arguments, we denote by $\vartheta(u)$ the unique argument of $u$ in $[-\pi, \pi)$. In particular, we notice that
\[
u =|u| {\rm e}^{{\rm i} \vartheta(u)}.
\]

\begin{theorem}
\label{th:finale}
Let $\{\vc d_1^0,\vc d_2^0\}$ and $\{\widetilde{\vc d}_1^0,\widetilde{\vc d}_2^0\}$ be such that $\{\vc t(0),\vc d_1^0,\vc d_2^0\}$ and $\{\vc t(0),\widetilde{\vc d}_1^0,\widetilde{\vc d}_2^0\}$ are two orthonormal basis in $\R^3$. Let $(u_1,u_2)$ and $(\widetilde{u}_1, \widetilde{u}_2)$ be the respective solutions of \eqref{eq:k1} and \eqref{eq:k2}. Then, the following hold true
\begin{equation}
    \label{eq:abs}
    \abs{u} = \abs{\widetilde{u}}
\end{equation}
and
\begin{equation}
    \label{eq:theta}
    \vartheta(u)- \vartheta(\tilde{u}) = \alpha,
\end{equation}
where $\alpha$ is a constant independent of $s$.
\end{theorem}
\begin{proof}
  First of all, to show the first relation \eqref{eq:abs}, we notice, by \eqref{eq:curva_antman}$_1$, that
\[
\abs{u}^2 = u_1^2+u_2^2 = |\vc t'|^2 = \widetilde{u}_1^2 + \widetilde{u}_2^2 = \abs{\widetilde{u}}^2.
\]
Next, to verify \eqref{eq:theta}, we observe directly by Proposition \ref{th:integral} that
\[
u_1 = -\vc t' \cdot \vc d_2, \quad u_2 = \vc t' \cdot \vc d_1,
\]
from which we get 
\[
u = \vc t'(s) \cdot \left(\vc d_1 - {\rm i} \,\vc d_2\right).
\]
Moreover, from Proposition \ref{prop:esistenzaRPAF}, we can write the rotated frame $(\widetilde{\vc d_1}, \widetilde{\vc d_2})$ in the plane perpendicular to $\vc t$ as 
$$
\left\{
\begin{aligned}
    &\widetilde{\vc d_1} = A \vc d_1 + B \vc d_2\\
    &\widetilde{\vc d_2} = \Gamma \vc d_1 + \Delta \vc d_2
\end{aligned}
\right.
$$
where $A, B, \Gamma, \Delta$ are constants. Then, it can be easily seen that
\[
    \tilde{u} = \tilde{u}_2+ {\rm i} \tilde{u}_1 = \vc t' \cdot \left(A \vc d_1 + B \vc d_2(s) - {\rm i} \Gamma \vc d_1 - {\rm i} \Delta \vc d_2\right)= \left(A u_2 - B u_1\right) + {\rm i} \left(\Delta u_1 - \Gamma u_2\right).
\]
Since the matrix $$
\begin{pmatrix}
    A&B\\
    \Gamma &\Delta
\end{pmatrix}
$$
is a constant rotation, it follows that $\tilde{u} = u {\rm e}^{{\rm i} \alpha}$ where the angle $\alpha$ depends only on $A, B, \Gamma, \Delta$, and this yields the conclusion.
\end{proof}


\section{Conclusions and remarks}
\label{sec:conclusion}
In this work, we compared the Serret-Frenet approach with the RPAF one.

In the first, starting from a moving frame $\{\vc t,\vc n,\vc b\} \in (W^{1,2}((0,L);\R^3))^3$ such that $\vc t' \cdot \vc b=0$, we can define the curvature and the torsion of $\vc x$ as follows
\[
\begin{aligned}
    &\kappa = \vc t' \cdot \vc n, &&&\tau = \vc n' \cdot \vc b.
\end{aligned}
\]
We stress the fact that $\kappa$ and $\tau$ are always defined in a weak sense and they are $L^2$ functions.  However, in general, it is not true that starting from $\vc t \in W^{1,2}((0,L);\R^3)$ with $\abs{\vc t} = 1$, there exists a moving frame $\{\vc t, \vc n, \vc b\} \in (W^{1,2}((0,L);\R^3))^3$ satisfying
\begin{equation}
\label{eq:SF_final}
  \left\{
\begin{aligned}
&\vc x'=\vc t,\\
&\vc{t}' = \kappa \vc{n},\\
&\vc{n}' = -\kappa \vc{t} + \tau \vc{b},\\
&\vc{b}' = -\tau \vc{n}.
\end{aligned}
\right. 
\end{equation}
Indeed, to have $\vc n \in W^{1,2}((0,L);\R^3)$, we must require, whenever $\abs{\vc t'} \neq 0$
\begin{equation}
\label{eq:regolarita}
    \frac{\vc t'}{\abs{\vc t'}} \in W^{1,2}((0,L);\R^3).
\end{equation}
The condition \eqref{eq:regolarita} could not be true without further assumptions on $\vc t$: observe that in general we cannot say more than $\vc t' \in L^2([0,L];\R^3)$. 

For a RPAF system, we immediately notice that the moving frame  $\{\vc t,\vc n,\vc b\}$ is not a RPAF since $\vc n'$ and $\vc b'$ are not parallel to $\vc t$. By means of Remark \ref{existence_relatively_parallel}, on a curve a RPAF generated by $\vc t \in W^{1,2}((0,L);\R^3)$ always exists. In this sense, the RPAF approach is more general since we can deal with any curve $\vc x \in W^{2,2}((0,L);\R^3)$.
Nevertheless, 
by Theorem \ref{th:finale}, the curvature and the torsion are defined as (see \eqref{eq:abs} and \eqref{eq:theta})
\[
\kappa = \sqrt{u_1^2+u_2^2} \qquad \text{and} \qquad \tau = \vartheta',
\]
where $u_2+ {\rm i}u_1 = \kappa {\rm e}^{{\rm i}\vartheta}$ and $u_1, u_2$ are the solutions of \eqref{eq:k1} and \eqref{eq:k2}. The main drawback of this approach stems in $\tau$ which is defined only in the sense of distributions: choosing discontinuous coefficients $u_1$ and $u_2$, we indeed get the angle $\vartheta$ to be a discontinuous function. We remark that this fact cannot happen for a frame of type \eqref{eq:SF_final}, where $\tau \in L^2((0,L))$.

To study variational problems for elastic curves related to functionals of type
\[
    \mathcal{F}[\vc x] = \int_{\vc x} f(\kappa, \tau)\, d \ell,
\]
one needs to introduce a weak notion of curvature and torsion. For instance, in \cite{BLMPRSA, BBLM}, we used the framed curve approach. Precisely, we considered functionals of the following type 
\[
\mathcal{F}[\vc t|\vc n|\vc b] = \int_0^L f(\vc t'\cdot \vc n, \vc n'\cdot \vc b)\, ds,
\]
where the independent variable is the moving frame $\{\vc t, \vc n, \vc b\}$. On the other hand, the formulation of the functional $\mathcal{F}$ in terms of RPAF is harder since the torsion is defined only in the sense of distributions. 

As a conclusion, it seems that the RPAF's approach is not suitable for performing a variational analysis of a functional depending on curvature and torsion. Nevertheless, it could be a useful approach to study geometric properties of curves since curvature and torsion turn out to be well-defined in a weaker framework.

\section*{Acknowledgements}
The authors thank Marco Degiovanni and Giulio Giusteri and for helpful suggestions and fruitful discussions.

GB is supported by the European Research Council (ERC), under the European Union's Horizon 2020 research and innovation programme, through the project ERC VAREG - {\em Variational approach to the regularity of the free boundaries} (grant agreement No. 853404). GB and LL are supported by italian Gruppo Nazionale per l'Analisi Matematica, la Probabilità e le loro Applicazioni (GNAMPA) of Istituto Nazionale per l'Alta Matematica (INdAM). AM is supported by italian Gruppo Nazionale per la Fisica Matematica (GNFM) of Istituto Nazionale per l'Alta Matematica (INdAM).

\bibliographystyle{abbrv}
\bibliography{references}

\end{document}